\documentclass[a4,reqno,12pt]{amsart}
\usepackage{amsmath,amsthm,amssymb,color}

\makeatletter
\@namedef{subjclassname@2010}{%
  \textup{2010} Mathematics Subject Classification}
\makeatother

\newcommand{\C}{\mathbb{C}}
\newcommand{\U}{\mathcal{U}}
\newcommand{\n}[1]{\left\vert {#1} \right\vert}                    
\newcommand{\N}[1]{\left\Vert {#1} \right\Vert}                    
\newcommand{\inner}[2]{\left\langle {#1} , {#2} \right\rangle}     
\newcommand{\Vol}[1]{{\rm Vol} \left( {#1} \right)}               
\newcommand{\Volb}[1]{{\rm Vol}_{\alpha} \left( {#1} \right)}               
\newcommand{\cent}{t}                
\newcommand{\BN}[1]{\left\Vert {#1} \right\Vert_{\alpha}}                      
\newcommand{\Binner}[2]{\left\langle {#1} , {#2} \right\rangle_{\alpha}} 
\newcommand{\BKer}[2]{K^{(\alpha)}_{\mathcal{U}} ( {#1} , {#2}) }
\newcommand{\BKalp}[1]{K^{(\alpha)}_{{#1}} }

\begin{document}

\newtheorem{Def}{Definition}[section]
\newtheorem{Thm}[Def]{Theorem}
\newtheorem{Lem}[Def]{Lemma}
\newtheorem{IntroThm}{Theorem}
\newtheorem{IntroCor}[IntroThm]{Corollary}
\newtheorem{Prop}[Def]{Proposition}
\newtheorem{Cor}[Def]{Corollary}
\newtheorem*{Cor*}{Corollary}
\renewcommand{\theIntroThm}{\Alph{IntroThm}}

 \numberwithin{equation}{section}

\title[Essential norm estimates for Toeplitz operators]
{Essential norm estimates for positive Toeplitz operators on the weighted Bergman space of a minimal bounded homogeneous domain}


\author[S. Yamaji]{Satoshi Yamaji}
\address{%
Satoshi Yamaji \endgraf
Graduate School of Mathematics \endgraf
Nagoya University \endgraf
Chikusa-ku, Nagoya, 464-8602 \endgraf
Japan
}

\email{satoshi.yamaji@math.nagoya-u.ac.jp}

\maketitle

\begin{abstract}
We give estimates for the essential norms of a positive Toeplitz operator on the Bergman space of a minimal bounded homogeneous domain in terms of the Berezin transform or the averaging function of the symbol. 
Using these estimates, we also give necessary and sufficient conditions for the positive Toeplitz operators to be compact. 
\end{abstract}

\section{Introduction}
The essential norm $\N{T}_e$ of a bounded operator $T$ is defined by 
\begin{align*}
\N{T}_e := \inf \{ \N{T-K} ; K {\rm \ is \ compact} \} .
\end{align*}
It is easy to see that $T$ is compact if and only if $\N{T}_e=0$. 
Essential norm estimates for Toeplitz operators with symbol in $L^{\infty}$ are considered in \cite{Zheng}, \cite{Choe}. 
In the present paper, we give an estimate of the essential norms of bounded positive Toeplitz operators on the weighted Bergman space of a minimal bounded homogeneous domain in terms of the Berezin transform or the averaging function. 

In 1988, Zhu \cite{Zhu1} obtained conditions in order that a positive Toeplitz operator is bounded or compact on the weighted Bergman space of a bounded symmetric domain in its Harish-Chandra realization. 
He characterized the conditions by using the Carleson type measures, the averaging function, and the Berezin transform. 
In \cite{Yamaji}, we consider the same problem for the Bergman space of a minimal bounded homogeneous domain. 
On the other hand, \v{C}u\v{c}kovi\'{c} and Zhao \cite{CZ} gave estimates for the essential norms of weighted composition operators by using the Berezin transform. 
We apply their methods for the case of essential norm estimates for positive Toeplitz operators 
on the weighted Bergman space of a minimal bounded homogeneous domain. 
Essential norm estimate tells us necessary and sufficient conditions for the positive Toeplitz operators to be compact. 
We obtain \cite[Theorem A]{Zhu1} and \cite[Theorem 1.2]{Yamaji} from the main theorem of this paper. 

Let $\U \subset \C^d$ be a minimal bounded homogeneous domain with a center $\cent \in \U$ 
(for the definition of the mininal domain, see \cite{I-Y}, \cite{MM}). 
For example, the unit ball, bounded symmetric domain in its Harish-Chandra realization, and 
representative domain of bounded homogeneous domains are 
minimal bounded homogeneous domains with a center $0$. 
Moreover, the products of these domains are also minaimal domains. 
We fix a minimal bounded homogeneous domain $\mathcal{U}$ with a center $\cent$. 
Let $dV(z)$ be the Lebesgue measure, $\mathcal{O}(\U)$ the space of all holomorphic functions on $\U$, 
and $L^p_a(\U)$ the Bergman space $L^p(\U,dV) \cap \mathcal{O}(\U)$ of $\U$. 
We denote by $K_{\U}$ the Bergman kernel of $\U$, that is, the reproducing kernel of $L^2_a(\U,dV)$. 
For $\alpha \in \mathbb{R}$, let $dV_{\alpha}$ denote the measure on $\U$ given by $dV_{\alpha}(z) := K_{\mathcal{U}}(z,z)^{-\alpha} dV(z)$.
It is known that there exists a constant $\varepsilon_{{\rm min}}$ such that the weighted Bergman space 
$L^2_a(\U,dV_{\alpha}) := L^2(\mathcal{U},dV_{\alpha}) \cap \mathcal{O}(\mathcal{U})$ 
is non-trivial if and only if $\alpha > \varepsilon_{{\rm min}}$ (for explicit expression of $\varepsilon_{{\rm min}}$, see \cite[(5.1)]{YamajiComp}). 
For $\alpha > \varepsilon_{{\rm min}}$, we consider the weighted Bergman space 
$ L^p_a(\mathcal{U},dV_{\alpha}) := L^p(\mathcal{U},dV_{\alpha}) \cap \mathcal{O}(\mathcal{U})$.

Let $\mu$ be a Borel measure on $\mathcal{U}$. 
For  $f \in L^2_a(\mathcal{U},dV_{\alpha})$, the Toeplitz operator $T_{\mu}$ with symbol $\mu$ is defined by 
\begin{align*}
T_{\mu} f(z) := \int_{\mathcal{U}} K_{\U}^{(\alpha)}(z,w) f(w) \, d\mu(w) \ \ \ (z \in \mathcal{U}) .
\end{align*}
If $d\mu(w) = u(w) dV_{\alpha}(w)$ holds for some $u \in L^{\infty}(\mathcal{U})$, we have $T_{\mu} f =P(uf)$, 
where $P$ is the orthogonal projection from $L^2(\mathcal{U},dV_{\alpha})$ onto $L^2_a(\mathcal{U},dV_{\alpha})$. 
A Toeplitz operator is called positive if its symbol is positive. 
Throughout this paper, we assume that $\mu$ is a positive Borel measure on $\mathcal{U}$.

Since the Bergman kernel of a minimal bounded homogeneous domain satisfies a useful estimate 
(see \cite[Theorem A]{I-Y}), the boundedness of $T_{\mu}$ on $L^2_a(\mathcal{U},dV_{\alpha})$ is characterized by using the Carleson type measures, the averaging function, and the Berezin transform (the definitions of them, see section \ref{sect200}).
We obtain the following theorem from the boundedness of the positive Bergman projection and Zhu's method (see \cite{Zhu1} or \cite{Zhu3}).  
\begin{IntroThm} \label{Bddness}
The following conditions are all equivalent.\\
$(a)$ \ $T_{\mu}$ is a bounded operator on $L^2_a(\mathcal{U},dV_{\alpha})$.\\
$(b)$ \ The Berezin transform $\widetilde{\mu}(z) $ is a bounded function on $\mathcal{U}$.\\
$(c)$ \ For all $p > 0$, $\mu$ is a Carleson measure for $L^p_a(\mathcal{U},dV_{\alpha})$. \\
$(d)$ \ The averaging function $\widehat{\mu}(z) $ is bounded on $\mathcal{U}$.
\end{IntroThm}

Theorem \ref{Bddness} is same as \cite[Theorem A]{Zhu1} if $\mathcal{U}$ is 
a Harish-Chandra realization of bounded symmetric domain 
and as \cite[Theorem 1.2]{Yamaji} if $\alpha =0$, that is, the case of the non-weighted Bergman space. 

We estimate the essential norm of the bounded positive Toeplitz operators on $L^2_a(\mathcal{U},dV_{\alpha})$. 
\begin{IntroThm} \label{EssEst}
Assume that $T_{\mu}$ is a bounded operator on $L^2_a(\mathcal{U},dV_{\alpha})$. Then, one has
\begin{align*}
 \N{T_{\mu}}_e \sim
 \limsup_{z \rightarrow \partial \U}  \widetilde{\mu}(z)  \sim
     \limsup_{z \rightarrow \partial \U} \widehat{\mu}(z) ,
\end{align*}
where the notation $\sim$ means that the ratios of the two terms are bounded below and above by constants. 
\end{IntroThm}

Since $T_{\mu}$ is compact if and only $\N{T_{\mu}}_e=0$, Theorem \ref{EssEst} yields the following corollary. 
\begin{IntroCor} \label{Intro3}
Let $\mu$ be a finite positive Borel measure on $\mathcal{U}$. 
Then the following conditions are all equivalent.\\
$(a)$ \ $T_{\mu}$ is a compact operator on $L^2_a(\mathcal{U},dV_{\alpha})$.\\
$(b)$ \ $\widetilde{\mu}(z)$ tends to $ 0$ as $z \rightarrow \partial \mathcal{U}$.\\ 
$(c)$ \ For all $p > 0$, $\mu$ is a vanishing Carleson measure for $L^p_a(\mathcal{U},dV_{\alpha})$. \\
$(d)$ \ $\widehat{\mu}(z)$ tends to $0$ as $z \rightarrow \partial \mathcal{U}$. 
\end{IntroCor}
Corollary \ref{Intro3} is also a generalization of \cite[Theorem B]{Zhu1} or \cite[Theorem 1.3]{Yamaji}.

Let us explain the organization of this paper. 
Section \ref{sect200} is preliminaries. 
In \ref{sect210}, we review properties of the weighted Bergman spaces 
of a minimal bounded homogeneous domain. 
Proposition \ref{NormalBK} plays an important role in the lower estimate of $\N{T_{\mu}}_e$. 
In \ref{sect220} and \ref{sect230}, we recall the definitions of Berezin symbol, averaging function, 
Carleson measure and vanishing Carleson measure. 
In \ref{sect240}, we prove the boundedness of the positive Bergman operator $P^+_{\U}$ on $L^2(\mathcal{U},dV_{\alpha})$ by using the boundedness of $P^+_{\mathcal{D}}$, where $\mathcal{D}$ is a Siegel domain biholomorphic to $\U$. 
The boundedness of $P^+_{\mathcal{D}}$ is given by B{\'e}koll{\'e} and Kagou \cite{B-K}. 
In section \ref{sect300}, we show necessary and sufficient conditions for positive Toeplitz operators on the weighted Bergman space of a minimal bounded homogeneous domain to be bounded (Theorem \ref{Bddness}). 
The boundedness of $P^+_{\U}$ yields an equality (Proposition \ref{IntForm}). 
This equality means that the inner product of $T_{\mu}f$ and $g$ in $L^2_a(\mathcal{U},dV_{\alpha})$ is equal to the inner product of $f$ and $g$ in $L^2_a(\mathcal{U},d \mu)$. 
In section \ref{sect400}, we estimate the essential norm of $T_{\mu}$. 
We use the sequence $\{ k_z^{(\alpha)} \}_{z \rightarrow \partial \U}$ to obtain the lower estimate of $\N{T_{\mu}}_e$. 
In section \ref{sect420}, we prove the upper estimate of $\N{T_{\mu}}_e$. 
We consider the estimate of the operator norm of $T_{\mu}(I-Q_n)$, where $Q_n$ is a compact operator defined from an orthonormal basis of $L^2_a(\mathcal{U},dV_{\alpha})$. 
By Proposition \ref{IntForm}, we consider the $L^2(d\mu)$-norm of $(I-Q_n)f$ and 
$T_{\mu}(I-Q_n)f$. Then, a property of Berezin transform (Lemma \ref{Mu}) plays an important role. \\

Throughout the paper, $C$ denotes a positive constant whose value may change from one occurrence to the next one.

\section{Preliminaries}  \label{sect200}
\subsection{Weighted Bergman space of a minimal bounded homogeneous domain}  \label{sect210}
For a bounded domian $\U \subset \C^{d}$, it is known that $\U$ is a minimal domain with a center $t \in \U$ 
if and only if 
$$K_{\U}(z,t) = \frac{1}{\Vol{\U}}$$
for any $z \in \U$ (see \cite[Proposition 3.6]{I-K} or \cite[Theorem 3.1]{MM}). 
For any $z \in \mathcal{U}$ and $r>0$, let
$$  B(z,r) := \{ w \in \mathcal{U} \mid d_{\U} (z,w) \leq r \} $$
be the Bergman metric disk with center $z$ and radius $r$, 
where $d_{\U} (\cdot,\cdot)$ denotes the Bergman distance on $\mathcal{U}$. 

For $f \in L^2_a(\mathcal{U},dV_{\alpha})$, we write 
\begin{align}
\N{f}_{\alpha} := \left( \int_{\U} \n{f(z)}^2 \, dV_{\alpha}(z) \right)^{\frac{1}{2}}
\end{align}
and we denote by $\inner{\cdot}{\cdot}_{\alpha}$ the inner product of $L^2_a(\mathcal{U},dV_{\alpha})$. 
We denote by $K_{\mathcal{U}}^{(\alpha)}$ the reproducing kernel of $L^2_a(\U,dV_{\alpha})$. 
It is known that 
\begin{align}
K_{\mathcal{U}}^{(\alpha)}(z,w) = C_{\alpha} K_{\mathcal{U}}(z,w)^{1+\alpha}   \label{weightedB}
\end{align}
for some positive constant $C_{\alpha}$. 
For $z \in \mathcal{U}$, we denote by $k_z^{(\alpha)}$ the normalized reproducing kernel of $L^2_a(\U,dV_{\alpha})$, that is, 
$$   k_z^{(\alpha)} (w) := \frac{K_{\mathcal{U}}^{(\alpha)}(w,z)}{K_{\mathcal{U}}^{(\alpha)}(z,z)^{\frac{1}{2}}} 
    = \sqrt{C_{\alpha}} \left( \frac{K_{\mathcal{U}}(w,z)}{K_{\mathcal{U}}(z,z)^{\frac{1}{2}}} \right)^{1+\alpha} . $$
For any Borel set $E$ in $\U$, we define
$$  \Volb{E} := \int_{E} dV_{\alpha}(w) . $$

First, we prove the following lemma. Although the proof is same as the one for the case of other domains (see \cite{CZ}, \cite{Eng}), we write it here for the sake of completeness. 
\begin{Lem} \label{WeakUnif}
A sequence of functions $\{ f_n \}$ in $L^2_a(\U,dV_{\alpha})$ converges to $0$ weakly in $L^2_a(\U,dV_{\alpha})$ if and only if $\{ f_n \}$ is bounded in $L^2_a(\U,dV_{\alpha})$ and converges to $0$ uniformly on each compact sets of $\U$. 
\end{Lem}
\begin{proof}
First, we prove the only if part. Suppose $\{ f_n \}$ converges to $0$ weakly in $L^2_a(\U,dV_{\alpha})$. 
Then, it is known that $\{ f_n \}$ is norm bounded in $L^2_a(\U,dV_{\alpha})$ and converges to $0$ pointwise. 
Take any subsequence of $\{ f_n \}$. Then, there exists a subsubsequence of $\{ f_n \}$ that converges to $0$ 
uniformly on each compact sets of $\U$ by Montel's theorem. 
Therefore, $\{ f_n \}$ itself converges to $0$ uniformly on each compact sets of $\U$. 

Next, we prove the if part. Suppose $\{ f_n \}$ is norm bounded and converges to $0$ uniformly 
on each compact sets of $\U$. Take any $\varepsilon >0$. 
For any $g \in L^2_a(\U,dV_{\alpha})$, 
there exists a compact set $K \subset \U$ such that 
\begin{align}
\int_{\U \backslash K} \n{g(z)}^2 \, dV_{\alpha}(z)  < \varepsilon . \label{Weak1}
\end{align}
Hence, we have 
\begin{align}
\n{\Binner{f_n}{g} } 
& \leq \n{\int_{K} f_n(z)\overline{g(z)} \, dV_{\alpha}(z) } + \n{\int_{\U \backslash K} f_n(z)\overline{g(z)} \, dV_{\alpha}(z) } \nonumber \\
& \leq \BN{g} \sup_{z \in K} \n{f_n(z)} + \BN{f_n} \int_{\U \backslash K} \n{g(z)}^2 \, dV_{\alpha}(z)  . \label{Weak2}
\end{align}
Since $\{ f_n \}$ converges to $0$ uniformly on $K$, there exists a $N \in \mathbb{N}$ such that 
the first term of (\ref{Weak2}) is less than or equal to $\BN{g} \varepsilon$ for any $n \geq N$. 
On the other hand, since $\{ f_n \}$ is norm bounded, there exists a $M>0$ such that $\Vert f_n \Vert_{\alpha} \leq M$. 
This together with (\ref{Weak1}), we see that the second term of (\ref{Weak2}) is less than or equal to $M\varepsilon$. 
Hence, we obtain $\langle f_n, g \rangle_{\alpha} \rightarrow 0$ as $n \rightarrow \infty$. 
This means that $\{ f_n \}$ converges to $0$ weakly in $L^2_a(\U,dV_{\alpha})$.
\end{proof}

In \cite{I-Y}, we obtain some estimates for the Bergman kernel of minimal bounded homogeneous domains. 
On the other hand, $K_{\mathcal{U}}^{(\alpha)}$ satisfies (\ref{weightedB}). 
Therefore, we have the following proposition. 
\begin{Prop} \label{NormalBK} 
$(i)$ \ For all compact set $K \subset \U$, there exists a constant $C>0$ such that $C^{-1} \leq \vert K_{\mathcal{U}}^{(\alpha)}(z,w) \vert \leq C$ 
for any $z \in K$ and $w \in \U$. \\
$(ii)$ \  A sequence $\{ k_a^{(\alpha)} \}$ converges to $0$ uniformly on each compact sets of $\U$ as $a \rightarrow \partial \mathcal{U}$. \\
$(iii)$ \ A sequence $\{ k_a^{(\alpha)} \}$ converges to $0$ weakly in $L^2_a(\mathcal{U},dV_{\alpha})$ as $a \rightarrow \partial \mathcal{U}$. 
\end{Prop}

\begin{proof}
First, we prove $(i)$. Take a compact set $K \subset \U$. 
Then, there exists a $\rho >0$ such that $K \subset B(t,\rho)$. 
By \cite[Proposition 6.1]{I-Y}, there exists a positive constant $M_{\rho}$ such that 
$M_{\rho}^{-1} \leq \n{K_{\mathcal{U}}(z,w)} \leq M_{\rho}$
 for any $z \in B(t,\rho)$ and $w \in \mathcal{U}$. 
This together with (\ref{weightedB}) follows from $(i)$. 
Next, we prove $(ii)$. Take a compact set $K \subset \U$ arbitrarily. 
From $(i)$, we obtain 
\begin{align*}
  \n{k_a^{(\alpha)}(z)} = \n{\frac{K_{\mathcal{U}}^{(\alpha)}(z,a)}{K_{\mathcal{U}}^{(\alpha)}(a,a)^{\frac{1}{2}}}}
      \leq \frac{C}{ K_{\mathcal{U}}(a,a)^{\frac{1+\alpha}{2}}}  
\end{align*}
for all $z \in K$ and $a \in \U$. 
Since $K_{\mathcal{U}}(a,a) \rightarrow \infty$ as $a \rightarrow \partial \mathcal{U}$ 
(see \cite[Proposition 5.2]{Koba}), we obtain $(ii)$. 
We also obtain $(iii)$ from Lemma \ref{WeakUnif}.
\end{proof}

\subsection{Berezin symbol and averaging function} \label{sect220}
For a Borel measure $\mu$ on $\mathcal{U}$, we define a function $\widetilde{\mu}$ on $\mathcal{U}$ by
\begin{align*}
 \widetilde{\mu}(z) := \int_{\mathcal{U}} \vert k_z^{(\alpha)}(w) \vert^2 \, d\mu(w) ,
\end{align*}
which is called the Berezin symbol of the measure $\mu$. 
Since $\n{K_{\mathcal{U}}(z,w)}$ is a bounded function on $B(\cent,r) \times \mathcal{U}$ 
by Proposition \ref{NormalBK}, 
$\widetilde{\mu}$ is a continuous function if $\mu$ is finite. 
For fixed $\rho>0$, we also define a function $\widehat{\mu}$ on $\mathcal{U}$ by
\begin{align*}
\widehat{\mu}(z) := \frac{\mu(B(z,\rho))}{\Volb{B(z,\rho)}} ,
\end{align*}
which is called the averaging function of the measure $\mu$. 
Although the value of $\widehat{\mu}$ depends on the parameter $\rho$, we will ignore that distinction. 

\subsection{Carleson measure and vanishing Carleson measure} \label{sect230}
Let $\mu$ be a positive Borel measure on $\mathcal{U}$ and $p >0$. 
We say that $\mu$ is a Carleson measure for $L^p_a(\mathcal{U},dV_{\alpha})$ if 
there exists a constant $M>0$ such that 
\begin{align*}
 \int_{\mathcal{U}} \vert f(z) \vert^p \, d\mu (z) \leq M \int_{\mathcal{U}} \vert f(z) \vert^p \, dV_{\alpha} (z) 
\end{align*}
for all $f \in L^p_a(\mathcal{U},dV_{\alpha})$. 
It is easy to see that $\mu$ is a Carleson measure for $L^p_a(\mathcal{U},dV_{\alpha})$ 
if and only if $L^p_a(\mathcal{U},dV_{\alpha}) \subset L^p_a(\mathcal{U}, d\mu)$ and the inclusion map 
$$  i_p : L^p_a(\mathcal{U},dV_{\alpha}) \longrightarrow  L^p_a(\mathcal{U}, d\mu)$$
is bounded. We see that the propery of being a Carleson measure is independent of $p$ (\cite[Theorem 3.2]{YamajiComp}). 

Suppose $\mu$ is a Carleson measure for $L^p_a(\mathcal{U},dV_{\alpha})$. 
We say that $\mu$ is a vanishing Carleson measure for $L^p_a(\mathcal{U},dV_{\alpha})$ if 
\begin{align*}
\lim_{k \rightarrow \infty} \int_{\U} \n{f_k(w)}^p \, d \mu(w) =0 
\end{align*}
whenever $\{ f_k \}$ is a bounded sequence in $L^p_a(\mathcal{U}, dV_{\alpha})$ that converges to $0$ 
uniformly on each compact subset of $\U$. 
We see that the propery of being a Carleson measure is also independent of $p$ (\cite[Theorem 3.3]{YamajiComp}).

\subsection{Boundedness of the positive Bergman operator}  \label{sect240}
In order to prove Theorem \ref{Bddness}, 
we use the boundedness of the positive Bergman operator $P_{\mathcal{U}}^+$ on $L^2(\mathcal{U},dV_{\alpha})$ defined by 
\begin{align}
P_{\mathcal{U}}^+ g(z) := \int_{\mathcal{U}} \n{K_{\U}^{(\alpha)}(z,w)} g(w) \, dV_{\alpha}(w)  \label{positive Bergman operator}
\end{align}
for $g \in L^2(\mathcal{U},dV_{\alpha})$. 
We prove that $P_{\mathcal{U}}^+$ is a bounded operator on $L^2(\mathcal{U},dV_{\alpha})$.

It is known that every bounded homogeneous domain is holomorphically equivalent to a homogeneous Siegel domain (see \cite{VGS}). 
Let $\Phi$ be a biholomorphic map from $\mathcal{U}$ to a Siegel domain $\mathcal{D}$. We define a unitary map $U_{\Phi}$ 
from $L^2(\mathcal{U},dV_{\alpha})$ to $L^2(\mathcal{D}, K_{\mathcal{D}}(\zeta,\zeta)^{-\alpha}dV(\zeta))$ by 
$$ U_{\Phi} f(\zeta)  := f(\Phi^{-1}(\zeta)) \n{\det J(\Phi^{-1},\zeta)}^{1+\alpha}  \ \ \ (f \in L^2(\mathcal{U},dV_{\alpha})). $$
Then, we have
$$    U_{\Phi} \circ P_{\mathcal{U}}^+ = P_{\mathcal{D}}^+ \circ U_{\Phi} . $$
Therefore, the boundedness of $P_{\mathcal{U}}^+$ on $L^2(\mathcal{U},dV_{\alpha})$ 
is equivalent to the boundedness of $P_{\mathcal{D}}^+$ on $L^2(\mathcal{D},K_{\mathcal{D}}(\zeta,\zeta)^{-\alpha}dV(\zeta))$. 
On the other hand, B{\'e}koll{\'e} and Kagou proved the boundedness of $P_{\mathcal{D}}^+$ (\cite[Theorem II.7]{B-K}). 
Therefore, we have the following lemma. 

\begin{Lem} \label{PosBergBddness} 
The operator $P_{\mathcal{U}}^+$ is bounded on $L^2(\mathcal{U},dV_{\alpha})$. 
\end{Lem}

\section{Boundedness of the Toeplitz operator}  \label{sect300}
In this section, we prove Theorem \ref{Bddness}. 
First, we prove the following proposition. 
\begin{Prop} \label{IntForm}
For some $p>0$, let $\mu$ be a Carleson measure for $L^p_a(\mathcal{U},dV_{\alpha})$. 
Then, $T_{\mu} f$ is in $L^2_a(\mathcal{U},dV_{\alpha})$ for any $f \in L^2_a(\mathcal{U},dV_{\alpha})$. 
Moreover, we have 
\begin{align}
\inner{T_{\mu} f}{g}_{\alpha} 
   = \int_{\mathcal{U}} f(w) \overline{g(w)} \, d\mu(w) \label{TmuKeisan}
\end{align}
for any $f, g \in L^2_a(\mathcal{U},dV_{\alpha})$.
\end{Prop}
\begin{proof}
For $f \in L^2_a(\mathcal{U},dV_{\alpha})$, we have 
\begin{align}
\BN{T_{\mu}f}^2 
               &= \int_{\mathcal{U}} \n{ \int_\mathcal{U} K_{\mathcal{U}}^{(\alpha)}(z,w) f(w) \, d\mu(w)}^2 dV_{\alpha}(z)  \nonumber \\
            &\leq \int_{\mathcal{U}} \left( \int_\mathcal{U} \n{K_{\mathcal{U}}^{(\alpha)}(z,w)} \n{ f(w)} \, d\mu(w) \right)^2 dV_{\alpha}(z) .\label{MainThm1siki}
\end{align}
Since $K_{\mathcal{U}}^{(\alpha)}(\cdot,z)f$ is in $L_a^1(\mathcal{U},dV_{\alpha})$ 
and $\mu$ is a Carleson measure, there exists a positive constant $M_{\mu}$ such that 
\begin{align}
\int_\mathcal{U} \n{K_{\mathcal{U}}^{(\alpha)}(z,w)} \n{ f(w)} \, d\mu(w)
  \leq  M \int_\mathcal{U} \n{K_{\mathcal{U}}^{(\alpha)}(z,w)} \n{ f(w)} \, dV_{\alpha}(w) . \label{MainThm1.5siki}
\end{align}
Note that $M_{\mu}$ is independent of $z$ by the definition of the Carleson measure. 
Therefore, we have 
\begin{align}
\BN{T_{\mu}f}^2
   \leq  M^2 \, \int_{\mathcal{U}} \left( \int_\mathcal{U} \n{K_{\mathcal{U}}^{(\alpha)}(z,w)} \n{ f(w)} \, dV_{\alpha}(w) \right)^2 dV_{\alpha}(z) \label{MainThm2siki}
\end{align}
by (\ref{MainThm1siki}) and (\ref{MainThm1.5siki}). 
Moreover, the right hand side of (\ref{MainThm2siki}) is equal to $M^2 \BN{P_{\mathcal{U}}^+ f^+}^2$, where $f^+:=\n{f}$. 
Since $P_{\mathcal{U}}^+$ is a bounded operator by Lemma \ref{PosBergBddness}, we have 
\begin{align}
\BN{T_{\mu}f} \leq M \BN{P_{\mathcal{U}}^+ f^+} \leq M \N{P_{\mathcal{U}}^+}\BN{f} . \label{TmuBdd} 
\end{align}
Next, we prove $T_{\mu}f \in \mathcal{O}(\mathcal{U})$. 
Since $T_{\mu}f \in L^2(\mathcal{U},dV_{\alpha})$, it is enough to prove $\inner{T_{\mu} f}{g}_{\alpha} =0$ for any $g \in L^2_a(\mathcal{U},dV_{\alpha})^{\perp}$. 
We have 
\begin{align}
\inner{T_{\mu} f}{g}_{\alpha} 
   &= \int_{\mathcal{U}} \left\{ \int_{\mathcal{U}} K_{\mathcal{U}}^{(\alpha)}(z,w) f(w) \, d\mu(w) \right\} \overline{g(z)} \, dV_{\alpha}(z) \nonumber \\
   &= \int_{\mathcal{U}} \overline{\left\{ \int_{\mathcal{U}} K_{\mathcal{U}}^{(\alpha)}(w,z) g(z)\, dV_{\alpha}(z) \right\} } f(w) \, d\mu(w)  \nonumber \\
   &= \int_{\mathcal{U}} \overline{P_{\mathcal{U}}g(w) } f(w) \, d\mu(w)   \label{FubiniSiyou}
\end{align}
for any $f, g \in L^2_a(\mathcal{U},dV_{\alpha})$. 
Note that since 
\begin{align}
 \int_{\mathcal{U}}  \int_{\mathcal{U}} \n{K_{\mathcal{U}}^{(\alpha)}(w,z) g(z) f(w)}\, d\mu(w) dV_{\alpha}(z)    
   \leq M \N{P_{\mathcal{U}}^+} \BN{f} \BN{g} < \infty ,  \label{Fubini0}
\end{align}
the second equality of $(\ref{FubiniSiyou})$ follows from Fubini's theorem. 
Therefore, (\ref{FubiniSiyou}) is equal to $0$ if $g \in L^2_a(\mathcal{U},dV_{\alpha})^{\perp}$. 

In view of (\ref{FubiniSiyou}) for the case that $g$ is in $L^2_a(\mathcal{U},dV_{\alpha})$, we obtain (\ref{TmuKeisan}). 
\end{proof}

By using Proposition \ref{IntForm}, we prove Theorem \ref{Bddness}. 
\begin{Thm}[Theorem \ref{Bddness}] \label{MainBddness}
Let $\mu$ be a positive Borel measure on $\mathcal{U}$. 
Then the following conditions are all equivalent.\\
$(a)$ \ $T_{\mu}$ is a bounded operator on $L^2_a(\mathcal{U},dV_{\alpha})$.\\
$(b)$ \ $\widetilde{\mu}(z) $ is a bounded function on $\mathcal{U}$.\\
$(c)$ \ For all $p >0$, $\mu$ is a Carleson measure for $L^p_a(\mathcal{U},dV_{\alpha})$. \\
$(d)$ \ $\widehat{\mu}(z) $ is a bounded function on $\mathcal{U}$.
\end{Thm}
\begin{proof}
For the equivaliance of (b), (c) and (d), see \cite[Theorem 3.3]{YamajiComp}. 
Moreover, we see that $(c) \Longrightarrow (a)$ follows from (\ref{TmuBdd}). 
Therefore, it is enough to prove $(a) \Longrightarrow (b)$. 
Since $T_{\mu}$ is a bounded operator, $T_{\mu}k_z^{(\alpha)}$ is in $L^2_a(\mathcal{U},dV_{\alpha})$. 
Therefore, we have 
\begin{align}
\Binner{T_{\mu}k_z^{(\alpha)}}{k_z^{(\alpha)}} 
    &=  \frac{T_{\mu}k_z^{(\alpha)}(z)}{\sqrt{\BKer{z}{z}}} \label{Bddness-1}
\end{align}
by reproducing property. The right hand side of (\ref{Bddness-1}) is equal to 
\begin{align*}
 \frac{1}{\sqrt{\BKer{z}{z}}}   \int_{\U} K(z,w) k_z^{(\alpha)}(w) \, d\mu (w)   
     = \widetilde{\mu}(z)
\end{align*}
Hence, we obtain
\begin{align*}
\n{\widetilde{\mu}(z)}  = \vert \inner{T_{\mu}k_z^{(\alpha)}}{k_z^{(\alpha)}}_{\alpha} \vert 
      \leq \N{T_{\mu}} \N{k_z^{(\alpha)}}_{2,\alpha}^2 = \N{T_{\mu}} < \infty .
\end{align*}
\end{proof}

\section{Essential norm estimates for the Toeplitz operator} \label{sect400}
In this section, we prove Theorem \ref{EssEst}. 
By \cite[(3.4)]{YamajiComp}, there exists a constant $C>0$ such that $\widehat{\mu}(z) \leq C \widetilde{\mu}(z)$ holds for any $z \in \U$. 
Therefore, it is enough to prove
\begin{align*}
 \limsup_{z \rightarrow \partial \U}   \widetilde{\mu}(z)
        \leq \N{T_{\mu} }_e \leq C  \limsup_{z \rightarrow \partial \U}   \widehat{\mu}(z) .
\end{align*}
\subsection{Lower estimates for the essential norm}  \label{sect410}
First, we prove the lower estimate of the essential norm of $T_{\mu}$. 
\begin{Thm} \label{LowerEst}
If $T_{\mu}$ is bounded, one has
\begin{align*}
 \limsup_{z \rightarrow \partial \U}   \widetilde{\mu}(z)
        \leq \N{T_{\mu} }_e . 
\end{align*}
\end{Thm}

\begin{proof}
Take a compact operator $K$ on $L^2_a(\mathcal{U},dV_{\alpha})$ arbitrarily. 
Since $\{ k_{z}^{(\alpha)} \}$ converges to $0$ weakly in $L^2_a(\mathcal{U},dV_{\alpha})$ as 
$z \rightarrow \partial \mathcal{U}$ (see Proposition \ref{NormalBK}), 
we have $\Vert {K k_{z}^{(\alpha)}} \Vert_{\alpha} \rightarrow 0$ as $z \rightarrow \partial \mathcal{U}$. 
Hence, we obtain
\begin{align}
\N{T_{\mu}  - K}  
\geq    \limsup_{z \rightarrow \partial \mathcal{U}}  \BN{(T_{\mu}  - K) k_{z}^{(\alpha)}}
\geq    \limsup_{z \rightarrow \partial \mathcal{U}}  \BN{T_{\mu}   k_{z}^{(\alpha)}} .  \label{Sita}
\end{align}
Since (\ref{Sita}) holds for every compact operator $K$, it follows that 
\begin{align}
\N{T_{\mu} }_e 
  \geq  \limsup_{z \rightarrow \partial \mathcal{U}} \BN{T_{\mu}   k_{z}^{(\alpha)}} . \label{Sita2}
\end{align}
On the other hand, since $T_{\mu}$ is bounded, we have 
\begin{align}
\widetilde{\mu}(z)  = \vert \inner{T_{\mu}k_z^{(\alpha)}}{k_z^{(\alpha)}}_{\alpha} \vert 
      \leq \N{T_{\mu}k_z^{(\alpha)}} .   \label{Sita3}
\end{align}
From (\ref{Sita2}) and (\ref{Sita3}), we complete the proof. 
\end{proof}
\subsection{Upper estimates for the essential norm}  \label{sect420}
In this section, we prove the upper estimate of the essential norm of $T_{\mu}$. 
From now on, we fix a constant $\rho >0$ and consider the case $r>\rho$. 
Let $\U_r := \U \backslash B(t,r)$ and $d\mu_r(z) := \chi_{\U_r}(z) d\mu(z)$, 
where $\chi_{\U_r}$ is the characteristic function on $\U_r$.  
First, we show the following lemma. 
\begin{Lem} \label{Mu}
If $T_{\mu}$ is bounded, one has
\begin{align*}
\sup_{z \in \U} \widehat{\mu_r}(z) \leq C \sup_{z \in \U_{r-\rho}} \widehat{\mu}(z) ,  
\end{align*}
where $C$ is a positive constant that is independent of $r$. 
\end{Lem}
\begin{proof}
By the definition of the averaging function and \cite[Lemma 2.3]{YamajiComp}, we have
\begin{align}
 \widehat{\mu_r}(z)
     &= \frac{1 }{\Volb{B(z,\rho)}} \int_{B(z,\rho) \cap \U_r}  \, d\mu(w)   \nonumber \\
     &\leq C\int_{B(z,\rho) \cap \U_r} \n{k_{z}^{(\alpha)}(w)}^2 \, d\mu(w) .   \label{WWW}
\end{align}
By \cite[Lemma 2.5]{YamajiComp}, we have 
\begin{align*}
\n{k_{z}^{(\alpha)}(w)}^2 \leq \frac {C}{\Volb{B(w,\rho)}} \int_{B(w,\rho)} \n{k_{z}^{(\alpha)}(u)}^2 \,  dV_{\alpha}(u) 
\end{align*}
for any $z \in B(u,\rho)$. Therefore, the right hand side of (\ref{WWW}) is less than or equal to 
\begin{align}
{} & C \int_{B(z,\rho) \cap \U_r} \int_{\U} 
          \frac {\chi_{B(w,\rho)}(u) \vert k_{z}^{(\alpha)}(u) \vert^2 }{\Volb{B(w,\rho)}} 
           \, dV_{\alpha}(u)  d\mu(w) \nonumber \\
    &= C \int_{\U}  \left\{ \int_{B(z,\rho) \cap \U_r}   
            \frac {\chi_{B(u,\rho)}(w) }{\Volb{B(w,\rho)}} d\mu(w) \right\}
             \n{k_{z}^{(\alpha)}(u)}^2 \,  dV_{\alpha}(u) \nonumber \\
    &\leq C   \sup_{u \in \U}  \left\{ \int_{B(z,\rho) \cap \U_r}   
            \frac {\chi_{B(u,\rho)}(w) }{\Volb{B(w,\rho)}} d\mu(w) \right\}
              \int_{\U} \n{k_{z}^{(\alpha)}(u)}^2 \,  dV_{\alpha}(u) \nonumber \\
   &= C   \sup_{u \in \U}  \left\{ \int_{B(z,\rho) \cap B(u,\rho) \cap \U_r}   
            \frac {1 }{\Volb{B(w,\rho)}} d\mu(w) \right\}   .  \label{WWW2}
\end{align}
Here, assume $w \in B(z,\rho) \cap B(u,\rho) \cap \U_r$. 
Then, since $d_{\U} (w,u) \leq \rho$ and $d_{\U}(t,w)>r$, we obtain 
\begin{align*}
d_{\U} (t,u) \geq d_{\U} (t,w) - d_{\U} (u,w) > r-\rho .
\end{align*}
Therefore, $u \in B(\cent, r-\rho)$. 
Hence, we obtain $ B(z,\rho)  \cap B(u,\rho) \cap \U_r= \emptyset$ for any $u \in \U \backslash \U_{r-\rho}$. 
Therefore, the right hand side of (\ref{WWW2}) is equal to 
\begin{align*}
 {}&  C   \sup_{u \in B(\cent, r-\rho)}  \left\{ \int_{B(z,\rho)  \cap B(u,\rho) \cap \U_r}   
              \frac {1 }{\Volb{B(w,\rho)}} d\mu(w) \right\}   \nonumber \\
    &  \leq  C   \sup_{u \in \U_{r-\rho}} \frac {1 }{\Volb{B(u,\rho)}} \left\{ \int_{B(z,\rho)  \cap B(u,\rho) \cap \U_r}   
               d\mu(w) \right\}   \nonumber \\
              &\leq  C   \sup_{u \in \U_{r-\rho}} \widehat{\mu}(u) .  
\end{align*}
Hence, we have 
$ \widehat{\mu_r}(z)  \leq  C   \sup_{u \in \U_{r-\rho}} \widehat{\mu}(u)$.
\end{proof}
Next, we denote the operators $Q_n$ and $R_n$ on $L^2_a(\U,dV_{\alpha})$. 
Suppose $\{ e_n \}$ is a complete orthonormal system of $L^2_a(\U,dV_{\alpha})$. 
For $n \in \mathbb{N}$ and $f \in L^2_a(\mathcal{U},dV_{\alpha})$, we denote $Q_n$ by
\begin{align*}
Q_n f(z) := \sum_{j=1}^{n} \Binner{f}{e_j} e_j(z) .
\end{align*}
The operator $Q_n$ is compact on $L^2_a(\U,dV_{\alpha})$. Let $R_n := I - Q_n$. 
Then we have 
$$\lim_{n \rightarrow \infty} \Vert R_n f \Vert_{\alpha} = 0$$
for each $f \in L^2_a(\U,dV_{\alpha})$.  
Moreover, we have $R_n^{\ast} = R_n$ and $R_n^{2} = R_n$. 

\begin{Lem}  \label{Lemma41}
If $T_{\mu}$ is bounded, one has
\begin{align*}
 \lim_{n \rightarrow \infty}  \sup_{\BN{f}=1} \N{T_{\mu}R_n  f}_{L^2(d\mu)}   \leq C \sup_{z \in \U_{r-\rho}} \widehat{\mu}(z) .
 \end{align*}
\end{Lem}
\begin{proof}
First, we prove 
\begin{align}
 \lim_{n \rightarrow \infty}  \sup_{\BN{f}=1} \int_{\U \backslash  {\U}_r} \n{T_{\mu}R_n  f(z)}^2 \, d\mu(z) =0.    \label{Suff1}
\end{align}
Since $T_{\mu}R_n  f \in L^2_a(\U,dV_{\alpha})$ by the boundedness of $T_{\mu}$, we obtain 
\begin{align*}
\n{T_{\mu}R_n  f (z)} =
\n{\Binner{T_{\mu}R_n  f}{\BKalp{z}}}
 = \n{\Binner{f}{R_n T_{\mu}^{\ast} \BKalp{z}}}
 \leq \BN{f} \BN{R_n T_{\mu}^{\ast} \BKalp{z}}  ,
\end{align*}
where the first equality follows from the reproducing property. 
Hence, we have 
\begin{align}
   \sup_{\BN{f}=1} \int_{\U \backslash  {\U}_r} \n{T_{\mu}R_n  f(z)}^2 \, d\mu(z) 
     \leq  \int_{\U \backslash  {\U}_r}   \BN{R_n T_{\mu}^{\ast} \BKalp{z}}^2 \, d\mu(z) .  \label{T1}
\end{align}
Here,we have
\begin{align*}
 \BN{R_n T_{\mu}^{\ast} \BKalp{z}}^2
 \leq  \N{T_{\mu}}^2 \BN{\BKalp{z}}^2  = \N{T_{\mu}}^2 \BKer{z}{z} . 
\end{align*}
Since  $\BKer{z}{z}$ is in $L^{\infty}(\U \backslash  {\U}_r)$, we have 
\begin{align}
 \lim_{n \rightarrow \infty} \int_{\U \backslash  {\U}_r}   \BN{R_n T_{\mu}^{\ast} \BKalp{z}}^2 \, d\mu(z) 
  =  \int_{\U \backslash  {\U}_r}    \lim_{n \rightarrow \infty} \BN{R_n T_{\mu}^{\ast} \BKalp{z}}^2 \, d\mu(z) 
 =0  \label{T2}
\end{align}
by Lebesgue's dominated convergence theorem. 
We obtain (\ref{Suff1}) from (\ref{T1}) and (\ref{T2}).

Next, we prove 
\begin{align}
\sup_{\BN{f}=1} \int_{ {\U}_r} \n{T_{\mu}R_n  f(z)}^2 \, d\mu(z)  \leq C \sup_{z \in \U_{r-\rho}} \widehat{\mu}(z)  .  \label{Suff2}
 \end{align}
By \cite[Lemma 3.1]{YamajiComp}, we have 
\begin{align*}
 \int_{ {\U}_r} \n{T_{\mu}R_n  f(z)}^2 \, d\mu(z) 
    &= \int_{ \U} \n{T_{\mu}R_n  f(z)}^2 \, d\mu_r(z)   \\
    &\leq C  \int_{\U} \widehat{\mu_r}(z) \n{T_{\mu}R_n  f(z)}^2 \, dV(z)  \\
    &\leq C \sup_{z \in \U} \widehat{\mu_r}(z) \BN{T_{\mu}R_n f}^2   \\  
    &\leq C \BN{f}^2  \sup_{z \in \U} \widehat{\mu_r}(z)  \\  
    &\leq C \BN{f}^2  \sup_{z \in \U_{r-\rho}} \widehat{\mu}(z)  , 
 \end{align*}
where the last inequality follows from Lemma \ref{Mu}. 
Therefore, we obtain (\ref{Suff2}). 

By (\ref{Suff1}) and (\ref{Suff2}), we have complete the proof. 
\end{proof}

Similarly, we have 
\begin{align}
 \lim_{n \rightarrow \infty}  \sup_{\BN{f}=1} \N{R_n  f}_{L^2(d\mu)}   \leq C \sup_{z \in \U_{r-\rho}} \widehat{\mu}(z) \label{Rnf}
\end{align}
if $T_{\mu}$ is bounded. 

\begin{Thm} \label{UpperEst}
If $T_{\mu}$ is bounded, one has
\begin{align*}
 \N{T_{\mu} }_e \leq  C \limsup_{z \rightarrow \partial \U}   \widehat{\mu}(z)       .
\end{align*}
\end{Thm}

\begin{proof}
Take any $n \in \mathbb{N}$. 
Since $Q_n$ is compact, $T_{\mu}  Q_n$ is also compact. 
Therefore, we have 
\begin{align*}
\N{T_{\mu}  }_e = \N{T_{\mu} R_n + T_{\mu}  Q_n }_e 
  = \N{T_{\mu}  R_n  }_e 
  \leq \N{T_{\mu}  R_n  }  . 
\end{align*}
Since 
\begin{align*}
\BN{T_{\mu}  R_n  f}^2  \leq  \N{R_n  f}_{L^2(d\mu)} \N{T_{\mu} R_n  f}_{L^2(d\mu)}
\end{align*}
by Proposition \ref{IntForm}, we have 
\begin{align*}
\N{T_{\mu}  }_e^2  \leq 
\N{T_{\mu}  R_n  }^2
  \leq \sup_{\BN{f}=1} \N{R_n  f}_{L^2(d\mu)}   \sup_{\BN{f}=1} \N{T_{\mu}R_n  f}_{L^2(d\mu)} . 
 \end{align*}
Take $n \rightarrow \infty$. Then, we obtain  
\begin{align*}
\N{T_{\mu}  }_e^2  \leq   C \left( \sup_{z \in \U_{r-\rho}} \widehat{\mu}(z)  \right)^2
\end{align*}
from Lemma \ref{Lemma41} and (\ref{Rnf}). 
Letting $r \rightarrow \infty$, we get
\begin{align*}
\N{T_{\mu}  }_e  \leq C  \limsup_{z \rightarrow \partial \U}  \widehat{\mu}(z)  . 
 \end{align*}
\end{proof}

We obtain Theorem \ref{EssEst} from Theorems \ref{LowerEst} and \ref{UpperEst}.
By using Theorem \ref{EssEst}, we obtain necessary and sufficient conditions for $T_{\mu}$ to be compact.
\begin{Cor} \label{MainCptness}  
Let $\mu$ be a finite positive Borel measure on $\mathcal{U}$. 
Then the following conditions are all equivalent.\\
$(a)$ \ $T_{\mu}$ is a compact operator on $L^2_a(\mathcal{U},dV_{\alpha})$.\\
$(b)$ \ $\widetilde{\mu}(z)  \rightarrow 0$ as $z \rightarrow \partial \mathcal{U}$.\\ 
$(c)$ \ For all $p>0$, $\mu$ is a vanishing Carleson measure for $L^p_a(\mathcal{U},dV_{\alpha})$. \\
$(d)$ \ $\widehat{\mu}(z)  \rightarrow 0$ as $z \rightarrow \partial \mathcal{U}$. 
\end{Cor}
\begin{proof}
By \cite[Theorem 3.3]{YamajiComp}, we have $(b) \Longleftrightarrow (c) \Longleftrightarrow (d)$. 
Therefore, it is enough to prove  $(a) \Longleftrightarrow (b)$. 
Assume $(a)$ holds. Then, $\N{T_{\mu}}_e=0$. 
Since a compact operator is bounded, we obtain $(b)$ by Theorem \ref{EssEst}. 
Next, assume $(b)$ holds. Since $\widetilde{\mu}$ is continuous, $\widetilde{\mu}$ is a bounded function. 
Therefore, $T_{\mu}$ is a bounded operator on $L^2_a(\mathcal{U},dV_{\alpha})$ by Theorem \ref{Bddness}. 
By using Theorem \ref{EssEst}, we obtain $(a)$. 
\end{proof}


\subsection*{Acknowledgment}
The author would like to thank to Professor H.~Ishi for useful advices 
and helpful discussions.

\end{document}